\documentclass[11pt]{article}
\usepackage{amsmath,amssymb,amsthm,xcolor,framed,esint,dsfont}
\usepackage[english]{babel}
\usepackage[margin=3cm]{geometry}
\usepackage{csquotes}
\usepackage{hyperref}
\usepackage{enumerate}
\usepackage{tikz}

\usepackage[title]{appendix}

\def\XXint#1#2#3{{\setbox0=\hbox{$#1{#2#3}{\int}$}
		\vcenter{\hbox{$#2#3$}}\kern-.5\wd0}}

\newcommand{\R}{{\mathbb R}}

\newcommand{\lt}{\left}
\newcommand{\rt}{\right}

\newcommand{\lm}{\lambda}
\newcommand{\nn}{\nonumber}

\newcommand{\ti}{\tilde}
\newcommand{\qd}{\quad}

\newtheorem{thm}{Theorem}[section]
\newtheorem{prop}[thm]{Proposition}
\newtheorem{lem}[thm]{Lemma}
\newtheorem{cor}[thm]{Corollary}
\theoremstyle{definition}
\newtheorem{rem}[thm]{Remark}
\newtheorem*{rem*}{Remark}

\newtheorem{defin}[thm]{Definition}

\numberwithin{equation}{section}

\title{Nonexistence of a class of $T_N$ configurations for a hyperbolic system with one entropy}
\author{Guanying Peng\footnote{Department of Mathematical Sciences, Worcester Polytechnic Institute, Worcester, MA 01609, USA. Email: \href{mailto:gpeng@wpi.edu}{gpeng@wpi.edu}} \and Anthony Vuolo\footnote{Department of Mathematics, Iowa State University, Ames, IA 50014, USA. Email: \href{mailto:vuolo@iastate.edu}{vuolo@iastate.edu}}}
\date{}

\begin{document}

\maketitle

\begin{abstract}
In \cite{KMS03study}, Kirchheim, M\"{u}ller and \v{S}ver\'{a}k proposed the program to use the differential inclusion approach to study entropy solutions for systems of conservation laws. In particular, they raised questions concerning the local structure of the rank-one convex hull of a set $K_a\subset\R^{3\times 2}$, which arises from the differential inclusion formulation of a classical $2\times 2$ system of conservation laws (the $p$-system) coupled with one entropy. Recently, this question has been studied extensively by showing that the set $K_a$ does not contain the so-called $T_N$ configurations for $N=4$ and $N=5$. In this paper, we continue this program by showing that the set $K_a$ does not contain a class of three-dimensional $T_N$ configurations, as well as two-dimensional $T_N$ configurations for general $N$.
\end{abstract}

\section{Introduction}
\label{sec_introduction}
The notion of entropy solutions as a selection criterion for physically relevant solutions has achieved great success for scalar conservation laws. However, for systems of conservation laws, the theory of entropy solutions (e.g., uniqueness) is much more limited. In \cite{KMS03study}, Kirchheim, M\"{u}ller and \v{S}ver\'{a}k pointed out a classical $2 \times 2$ system of conservation laws in one space dimension given by
\begin{equation}\label{eq201}
	\begin{cases}
	&v_t - u_x = 0, \\
	&u_t - a(v)_x = 0
	\end{cases}
\end{equation}
in the form of the $p$-system. Here $u$ and $v$ are the unknown functions, and $a\in C^2(\R)$ is a given function. This system has a natural physical entropy/entropy flux pair $(\eta, q)$ given by
\begin{equation*}
\eta(u,v) := \frac{1}{2}u^2 + F(v),\qquad q(u,v) := ua(v), 
\end{equation*}
where the function $F$ satisfies $F' = a$. In \cite{DP85compensated}, DiPerna studied bounded weak solutions to the system \eqref{eq201} satisfying the entropy condition
\begin{equation}\label{eq:entcondition}
	\eta(u,v)_t - q(u,v)_x \leq 0,
\end{equation}
and an additional entropy condition involving a dual entropy/entropy flux pair. Under the assumption that the function $a$ is strictly convex and increasing, i.e., the system \eqref{eq201} is strictly hyperbolic and genuinely nonlinear, DiPerna proved the local existence of solutions using the compensated compactness method. As a model problem to understand the compactness properties of entropy solutions for general systems, many of which possess only one entropy, the authors of \cite{KMS03study} raised questions concerning the local compactness of the differential inclusion
\begin{equation}\label{eq:diffinc}
    D\psi \in K_a \subset \mathbb{R}^{3\times2}
\end{equation}
for some vector-valued function $\psi$, where
\begin{equation}
    \label{eq_kspace}
        K_a:=\left\{\begin{bmatrix}
        u&v\\a(v)&u\\ua(v)&\frac{u^2}{2}+F(v)
    \end{bmatrix}:u,v\in \mathbb{R}\right\}.
    \end{equation}
This differential inclusion (at least locally) is a reformulation of the system \eqref{eq201} coupled with the entropy condition \eqref{eq:entcondition} with inequality replaced by equality (which is reasonable for studying compactness properties; see \cite[Section 7]{KMS03study} for more details). 

The structures of various convex hulls of a set $K \subset \mathbb{R}^{m \times n}$ play a crucial role in determining compactness properties of the differential inclusion $D\psi \in K$. At the heart of the different convexity notions is \emph{quasiconvexity}, which is closely related to the lower semicontinuity of multiple integrals; see \cite{Mo52quasi}. In a nutshell, the compactness properties of a differential inclusion $D\psi \in K$ are characterized by the structure of the quasiconvex hull of $K$. This is central to the compensated compactness method in the pioneering works \cite{tartar1979compensated,tartar1983compensated,DP85compensated}. In general, it is very challenging to characterize the structure of the quasiconvex hull of a set $K\subset\R^{m\times n}$, partly due to the fact that quasiconvex functions are very difficult to understand. Existing general methods for proving compactness of differential inclusions remain very limited; see \cite{Sv93tartar,FS2008tartar}. These works make use of the more transparent \emph{polyconvex} and \emph{rank-one convex} functions to aid in understanding the structure of the relevant quasiconvex hull. Specifically, the related  polyconvex hull and rank-one convex hull serve as upper and lower bounds of the quasiconvex hull, respectively. The reader is referred to \cite[Section 4]{KMS03study} for more detailed information on this topic.

For the set $K_a$ given in \eqref{eq_kspace}, Kirchheim, M\"{u}ller and \v{S}ver\'{a}k raised questions in \cite[Section 7]{KMS03study} concerning the local structures of the polyconvex hull and rank-one convex hull of $K_a$.
In \cite{LP19null}, among other things, the first author and Lorent showed that if the function $a$ is strictly increasing, then the polyconvex hull of $K_a$ is highly nontrivial, i.e., much larger than $K_a$. This opened the possibility that the quasiconvex hull of $K_a$ is also nontrivial, leading to lack of compactness for the differential inclusion \eqref{eq:diffinc} and the system \eqref{eq201} with the entropy condition \eqref{eq:entcondition}. In recent years, a number of striking examples have indicated some connections between lack of compactness and lack of uniqueness; see, e.g., \cite{KMS03study, DLS09euler, DLS10on, CDLK15global}. Consequently, understanding the compactness properties of the differential inclusion \eqref{eq:diffinc} is crucial for gaining insights into weak solutions of the system \eqref{eq201} satisfying the entropy condition \eqref{eq:entcondition}. 

The rank-one convex hull of $K_a$ provides a lower bound for its quasiconvex hull. The natural first step towards understanding the structure of the rank-one convex hull is to explore whether $K_a$ contains \emph{rank-one connections}. A set $K\subset\R^{m\times n}$ contains rank-one connections if there exist $A\ne B\in K$ such that $\mathrm{rank}(A-B) = 1$. If a set $K$ contains rank-one connections, then its rank-one convex hull is nontrivial and one can construct highly oscillatory approximate solutions to the differential inclusion into $K$ with no compactness; see, e.g., \cite[Lemma 3.2]{Ki03rigidity}. Under the assumption that the function $a$ is strictly convex and increasing, DiPerna pointed out in \cite{DP85compensated} that the set $K_a$ has no rank-one connections (an explicit proof was given in \cite[Proposition 4]{LP20on}). The next natural object that produces a nontrivial rank-one convex hull involves a special type of $N$-point configurations, called $T_N$ configurations (see Definition \ref{def_tn} in \S~\ref{ss:TN}). As noted in \cite{KMS03study}, the simplest $T_4$ configurations were discovered independently by a number of authors. Such $T_N$ configurations played fundamental roles in the convex integration construction of wild solutions to elliptic systems \cite{MS03convex,szekelyhidi2004regularity}, and in a large number of subsequent works, including  \cite{CFG11lack,FS2018T5,delellis2021geometric,HT21on,ST23the}. In \cite{LP20on}, the first author and Lorent proved that, if the function $a$ is strictly convex and increasing, then the set $K_a$ given in \eqref{eq_kspace} does not contain $T_4$ configurations. More recently, Johansson and Tione showed in \cite{JT24T5} nonexistence of $T_5$ configurations in the set $K_a$, and Krupa and Sz\'{e}kelyhidi established in \cite{KS24non}  nonexistence of $T_4$ configurations for a large class of sets containing the set $K_a$. Both of these works employed very different methods from those used in \cite{LP20on}. In the present paper, we continue this program by showing that the set $K_a$ does not contain a class of $T_N$ configurations for $N\geq 6$. 

To state our main result, we first introduce some notation. Let $P: \R^2\to \R^{3\times 2}$ be defined by
\begin{equation}
	\label{eq_p}
	P(u,v):=\begin{bmatrix}
		u&v\\a(v)&u\\ua(v)&\frac{u^2}{2}+F(v)
	\end{bmatrix}.
\end{equation}
For a set $\mathcal{X}=\{X_i\}_{i=1}^N\subset K_a$, let $X_i = P(u_i, v_i)$ for $i=1, \dots, N$. We introduce the matrix $A_{\mathcal{X}}\in\R^{3\times 2N}$ to be
 \begin{equation}
	\label{eq_atk}
	A_{\mathcal{X}}:=\begin{bmatrix}
		u_1&\null&u_N&v_1&\null&v_N\\a(v_1)&\cdots&a(v_N)&u_1&\cdots&u_N\\u_1 a(v_1)&\null&u_N a(v_N)&\frac{u_1^2}{2}+F(v_1)&\null&\frac{u_N^2}{2}+F(v_N)
	\end{bmatrix}.
\end{equation}
Further, for a matrix $Q\in\R^{3\times 2}$ with entries $q_{ij}$ for $1\leq i\leq 3, 1\leq j\leq 2$, we introduce the matrix $\Pi^N(Q)\in \R^{3\times 2N}$ to be
 \begin{equation}
	\label{eq_Pi}
	\Pi^{N}(Q):=\begin{bmatrix}
		q_{11}&\null&q_{11}&q_{12}&\null&q_{12}\\q_{21}&\cdots&q_{21}&q_{22}&\cdots&q_{22}\\q_{31}&\null&q_{31}&q_{32}&\null&q_{32}
	\end{bmatrix}.
\end{equation}
Then our main theorem is

\begin{thm}
	\label{thm_tn}
	Suppose $a\in C^2(\R)$ satisfies $a'>0$ and $a''>0$ on $\R$. Let $N\geq 6$ and $\mathcal{X}=\{X_i\}_{i=1}^N\subset K_a$ with $X_i = P(u_i, v_i)$ for $i=1, \dots, N$. If there exists $Q\in \R^{3\times 2}$ such that
	\begin{equation}\label{eq:r2assumption}
	 \mathrm{rank}\lt(A_{\mathcal{X}} - \Pi^N(Q)\rt)=2, 
	 \end{equation}
	 then $\mathcal{X}$ does not have an order that forms a $T_N$ configuration.
\end{thm}

Here the assumption on the function $a$ makes the system \eqref{eq201} strictly hyperbolic and genuinely nonlinear. If $a'<0$, i.e., the system \eqref{eq201} is elliptic, then the rank-one convex hull of $K_a$ is locally trivial as a consequence of \cite[Theorem 5]{LP19null}, and thus, at least locally, $K_a$ contains no $T_N$ configurations. On the other hand, if $a'>0$ and $a$ has an isolated inflection point, then the set $K_a$ contains rank-one connections by \cite[Proposition 4]{LP20on} (this was first observed by DiPerna in \cite{DP85compensated}). Further, the assumption \eqref{eq:r2assumption} is strongly motivated by the works \cite{LP20on,KS24non} on $T_4$ configurations. Specifically, in \cite{LP20on}, significant efforts were devoted towards ruling out $T_4$ configurations in $K_a$ satisfying \eqref{eq:r2assumption}. Later, it was observed in \cite[Lemma 4.2]{KS24non} that any $T_4$ configurations in $K_a$, if existed, would have to satisfy the condition \eqref{eq:r2assumption}. 

\begin{rem}
	Under the assumptions of Theorem \ref{thm_tn}, the set $\mathcal{X}$ does not have rank-one connections as a consequence of \cite[Proposition 4]{LP20on}. Further, $ \mathrm{rank}\lt(A_{\mathcal{X}} - \Pi^N(Q)\rt)$ cannot be 1; see the proof of Lemma \ref{l:linindependent} in \S~\ref{ss:prior}. In Appendix~\ref{sec:example}, we provide an explicit example of a function $a$ that satisfies the assumptions in Theorem \ref{thm_tn}, and a set $\mathcal{X}\subset K_a$ with at least 6 points that satisfies the assumption \eqref{eq:r2assumption}; see Lemma \ref{l:example}.
\end{rem}

As in \cite{LP20on}, Theorem \ref{thm_tn} allows us to rule out a special class of three-dimensional $T_N$ configurations (see Remark \ref{r:rank3}), as well as two-dimensional $T_N$ configurations in $K_a$. In particular, we have
\begin{cor}
	\label{thm_general2Dtn}
	Suppose $a\in C^2(\R)$ satisfies $a'>0$ and $a''>0$ on $\R$. Let $N\geq 6$ and $\mathcal{X}=\{X_i\}_{i=1}^N\subset K_a$. If there exists a two-dimensional plane $\mathcal{P}\subset\R^{3\times 2}$ such that $\mathcal{X}\subset K_a\cap \mathcal{P}$, then $\mathcal{X}$ does not have an order that forms a $T_N$ configuration.
\end{cor}

Our proof of Theorem \ref{thm_tn} builds on many ideas developed in \cite{LP20on}. Here we isolate the most essential ingredients from the $T_4$ case under the assumption \eqref{eq:r2assumption} in \cite{LP20on}, which allows us to extend these ideas to effectively rule out general $T_N$ configurations under the same assumption. Unfortunately, the assumption \eqref{eq:r2assumption} is crucial to our method of proof, and it remains unclear to us whether general $T_N$ configurations exist in the set $K_a$.

Recently, motivated by the Cauchy problem for the system \eqref{eq201}, Krupa considered the differential inclusion $D\psi\in\ti K_a$ in \cite{Kr24finite}, where $\ti K_a:=\left\{\begin{bmatrix}
	u&v\\a(v)&u
\end{bmatrix}\right\}\subset \R^{2\times 2}$, and, among other things, found a certain type of $T_5$ configurations (specifically, large $T_5$ configurations as explored in \cite{FS2018T5}) in the set $\ti K_a$. 

\bigskip

\noindent\textbf{Plan of the article.}
In \S~\ref{sec:prelim}, we introduce the definition of $T_N$ configurations and a necessary condition for $T_N$ configurations, as well as some useful results from \cite{LP20on}. In \S~\ref{sec:main}, we give the proof of Theorem \ref{thm_tn}. In \S~\ref{sec:proof}, we give the proof of Corollary \ref{thm_general2Dtn}. In Appendix~\ref{sec:example}, we provide an explicit example that satisfies the assumptions in Theorem \ref{thm_tn}.
\bigskip

\noindent\textbf{Acknowledgments.}
The authors were supported in part by the NSF grant DMS-2206291. The first author would like to thank Prof. Andrew Lorent, from whom he has learned everything related to the topics of this work.

\section{Preliminaries}\label{sec:prelim}

In this section, we gather some preliminary results that will be crucial for the proofs of Theorem \ref{thm_tn} and Corollary \ref{thm_general2Dtn} in the next two sections. 

\subsection{$T_N$ configurations and a necessary condition}\label{ss:TN}

We begin with the definition of $T_N$ configurations (see, e.g., \cite[Definition 1]{szekelyhidi2005rankone}).
\begin{defin}[$T_N$ configuration]
	\label{def_tn}
	An ordered set of $N\geq 4$ distinct matrices  $\{X_i\}_{i=1}^{N}\subset \mathbb{R}^{m\times n}$ without rank-one connections is said to form a $T_N$ configuration if there exist matrices $P, C_i\in\R^{m\times n}$ and real numbers $\kappa_i>1$ such that 
	\begin{align*}
			X_1&=P+\kappa_1 C_1,\\
			X_2&=P+C_1+\kappa_2 C_2,\\
			&\vdots \\
			X_N&=P+C_1+\dots +C_{N-1}+\kappa_N C_N,
	\end{align*}
	where $\mathrm{rank}(C_i)=1$ for all $i$ and $\sum_{i=1}^N C_i=0$.
\end{defin}


In \cite[Proposition 1]{szekelyhidi2005rankone}, Sz\'{e}kelyhidi provided a very useful algebraic characterization for $T_N$ configurations in $\R^{2\times 2}$. This characterization was extended to general $\R^{m\times n}$ in \cite[Proposition 3.6]{delellis2021geometric}. To avoid introducing unnecessary technical notation, we only state the following simple consequence of the above two results, which is what we need for the proof of Theorem \ref{thm_tn} (see also \cite[Proposition 5]{LP20on} and \cite[Proposition 1.3]{KS24non}).
\begin{prop}[\cite{szekelyhidi2005rankone,delellis2021geometric}]
	\label{def_tn_prop}
	Let $Z$ be a subset of $\{1,2,3\}$ with cardinality two, and for any $M\in\R^{3\times 2}$, let $M^Z$ denote the $2\times2$ submatrix of $M$ formed from the two rows indexed by the elements of $Z$. If a set $\{X_i\}_{i=1}^N\subset \mathbb{R}^{3\times 2}$ forms a $T_N$ configuration and the set $\{X_i^Z\}_{i=1}^N$ does not contain rank-one connections, then for every $i$, 
	\begin{equation}\label{eq:TNnecessary}
		\text{the set }\{\det(X_j^Z-X_i^Z): j\ne i\}\text{ necessarily changes sign}.
	\end{equation}
\end{prop}

\subsection{Some notation and results from \cite{LP20on}}\label{ss:prior}

In this subsection, we introduce some notation and results from \cite{LP20on}. 
For a set $\mathcal{X}=\{X_i\}_{i=1}^N\subset K_a$, let $X_i = P(u_i, v_i)$ for $i=1, \dots, N$, where recall that the mapping $P$ is defined in \eqref{eq_p}.
To simplify notation, we denote
\begin{equation}\label{eq:uva}
	\vec u:=(u_1,\dots, u_N), \qd\vec v:= (v_1,\dots, v_N),\qd \vec{\mathfrak{a}}:=(a(v_1),\dots,a(v_N)).
\end{equation}
We first introduce the following necessary condition for $\mathcal{X}$ to contain an order that forms a $T_N$ configuration. 
\begin{lem}\label{l:linindependent}
	Suppose $a\in C^2(\R)$ satisfies $a'>0$ and $a''>0$ on $\R$. Let $N\geq 6$ and $\mathcal{X}=\{X_i\}_{i=1}^N\subset K_a$ with $X_i = P(u_i, v_i)$. If $\mathcal{X}$ has an order that forms a $T_N$ configuration, then $\vec u$ and $\vec v$ are linearly independent, or $\vec u$ and $\vec{\mathfrak{a}}$ are linearly independent.
\end{lem}

This lemma is essentially a reformulation of \cite[Lemmas 11-12]{LP20on}. Here for the convenience of the reader, we give a much shorter proof.

\begin{proof}[Proof of Lemma \ref{l:linindependent}]
	From \cite[Proposition 3.2]{JT24T5}, if $\mathcal{X}$ has an order that forms a $T_N$ configuration, then there must exist $i\ne j$ and $k\ne\ell$ such that $u_i\ne u_j$ and $v_k\ne v_\ell$ (the statement for the former clearly follows from the proof of \cite[Proposition 3.2]{JT24T5}). In particular, $\vec u$, $\vec v$ and $\vec{\mathfrak a}$ are all nonzero vectors in $\R^N$. Suppose, by contradiction, that there exist $\alpha, \beta\ne 0$ such that $\vec u = \alpha \vec v$, $\vec{\mathfrak{a}} = \beta \vec u$. Then $\vec{\mathfrak{a}} = \alpha\beta\vec v$, i.e., $a(v_i) = \alpha\beta v_i$ for all $i=1,\dots, N$. Since $a$ is strictly convex, the equation $a(t) = \alpha \beta t$ has at most two distinct solutions, say $t_1<t_2$. Then $(u_i, v_i)\in\{(\alpha t_1, t_1), (\alpha t_2, t_2)\}$, which would imply $\mathrm{Card}(\mathcal{X})\leq 2$ and contradict to $N\geq 6$.
\end{proof}

Our strategy of the proof of Theorem \ref{thm_tn} is to show that there exists $i$ such that the condition \eqref{eq:TNnecessary} is violated. To this end, for every $i=1,\dots, N$, it is convenient to consider an equivalent problem which essentially translates $X_i$ to the zero matrix. Specifically, we denote
\begin{equation}\label{eq:hiri}
	h_j^i:= u_j-u_i, \quad r_j^i:= v_j - v_i\qquad\text{ for }j=1, \dots, N.  
\end{equation}
Further, for $i=1, \dots, N$, define
\begin{equation}\label{eq:aiFi}
	a^i(t) := a(t+v_i) - a(v_i),\qd F^i(t):= F(t+v_i) - F(v_i) - a(v_i)t.  
\end{equation}
Since $a'>0$, $a''>0$ and $F' = a$, it is clear that
\begin{equation}\label{eq:aiproperty}
	a^i(0)=0, \qd (a^i)' >0, \qd (a^i)''>0,
\end{equation}
and
\begin{equation}\label{eq:Fiproperty}
	(F^i)' = a^i, \qd F^i(0)=(F^i)'(0)=0. 
\end{equation}
Lastly, define
\begin{equation}\label{eq:TXi}
	\mathbb{T}_{\mathcal{X}}^i := \{T^i_1,\dots,T^i_N\}\subset \R^{3\times 2},
\end{equation}
where
\begin{equation*}
	T^i_j:=\begin{bmatrix}
		h_j^i&r_j^i\\a^i(r_j^i)&h_j^i\\h_j^ia^i(r_j^i)&\frac{(h_j^i)^2}{2}+F^i(r_j^i)
	\end{bmatrix}.
\end{equation*}
The following lemma is essentially \cite[Lemma 9]{LP20on}, whose proof trivially generalizes to $T_N$ configurations.

\begin{lem}[$\text{\cite[Lemma 9]{LP20on}}$]\label{l:TXi}
	Let $\mathcal{X}=\{X_i\}_{i=1}^N\subset K_a$ with $X_i = P(u_i, v_i)$. If $\mathcal{X}$ forms a $T_N$ configuration, then $\mathbb{T}_{\mathcal{X}}^i$ also forms a $T_N$ configuration (not necessarily in $K_a$) for every $i=1,\dots,N$.
\end{lem}

\begin{rem}\label{r:linindependent}
	Since the function $a^i$ also satisfies $(a^i)'>0$ and $(a^i)''>0$ (see \eqref{eq:aiproperty}), the proof of Lemma \ref{l:linindependent} applies to the set $\mathbb{T}_{\mathcal{X}}^i$. That is, denoting 
	\begin{equation*}
		\vec h^i := (h_1^i,\dots, h_N^i),\qd {\vec r}\,^i := (r_1^i,\dots, r_N^i),\qd \vec{\mathfrak{a}}\,^i:= (a^i(r_1^i),\dots,a^i(r_N^i)),
	\end{equation*}
	if $\mathbb{T}_{\mathcal{X}}^i$ has an order that forms a $T_N$ configuration, then $\vec h^i$ and ${\vec r}\,^i$ are linearly independent, or ${\vec{h}}^i$ and $\vec{\mathfrak{a}}\,^i$ are linearly independent. 
\end{rem}

The proof of Theorem \ref{thm_tn} involves carefully analyzing the structure of solutions to the system of equations given by
\begin{equation}\label{eq:rank2system}
	\begin{cases}
		&s a(t) = \lambda_1 s + \lambda_2 a(t),\\
		&\frac{s^2}{2}+ F(t) = \lambda_1 t + \lambda_2 s.
	\end{cases}
\end{equation} 
Here $s$ and $t$ are the unknowns and $\lambda_1, \lambda_2$ are two given constants. This system of equations arises as a consequence of the assumption \eqref{eq:r2assumption}, and has been studied in detail in \cite{LP20on}. First, we need the following lemma which deals with the above system \eqref{eq:rank2system} in the degenerate cases. This is covered in \cite[Lemma 19]{LP20on}, which requires additional efforts to deal with the particular $T_4$ case. For $T_N$ with $N>4$, we only need the following version, whose short proof is provided for the reader's convenience.
\begin{lem}[$\text{\cite[Lemma 19]{LP20on}}$]\label{l:degenerate}
	Suppose $a\in C^2(\R)$ satisfies $a'>0$ and $a''>0$ on $\R$, and $F\in C^3(\R)$ satisfies $F' = a$. In addition, assume $a(0) = F(0) = 0$. Then the system \eqref{eq:rank2system} has at most 4 distinct solutions if $\lambda_1=0$ or $\lambda_2=0$. 
\end{lem}
\begin{proof}
	We copy the system \eqref{eq:rank2system} here:
	\begin{equation}\label{eq:rank2system1}
		s a(t) = \lambda_1 s + \lambda_2 a(t),
	\end{equation} 
	\begin{equation}\label{eq:rank2system2}
		\frac{s^2}{2}+ F(t) = \lambda_1 t + \lambda_2 s.
	\end{equation}
	
	\emph{Case 1.} Suppose $\lambda_1=0$. Then \eqref{eq:rank2system1} reduces to $sa(t) = \lambda_2 a(t)$. If $a(t)\ne 0$, then $s=\lambda_2$. Plugging this into \eqref{eq:rank2system2}, we obtain $F(t) = \lambda_2^2/2$. Since $F'' = a' >0$, this equation has at most two distinct solutions. If $a(t)=0$, or equivalently $t=0$, then \eqref{eq:rank2system2} reduces to $s^2/2 = \lambda_2 s$, which also has at most two distinct solutions. Combining the two cases, we know the system \eqref{eq:rank2system} has at most four distinct solutions. 
	\medskip
	
	\emph{Case 2.} Suppose $\lambda_2=0$. Then \eqref{eq:rank2system1} reduces to $sa(t) = \lambda_1 s$. If $s\ne 0$, then $a(t)=\lambda_1$, which has at most one solution, say $t_0$, as $a'>0$. Plugging this into \eqref{eq:rank2system2}, we obtain $s^2/2 = a(t_0)t_0 - F(t_0)$, which has at most two distinct solutions for $s$. If $s=0$, then \eqref{eq:rank2system2} reduces to $F(t) = \lambda_1 t$, which also has at most two distinct solutions since the function $t\mapsto F(t)-\lambda_1t$ is strictly convex. Hence, the system \eqref{eq:rank2system} has at most four distinct solutions in this case as well.
\end{proof}

Finally, we put together the relevant results from \cite{LP20on} concerning the system \eqref{eq:rank2system} in the nondegenerate cases.

\begin{lem}[$\text{\cite[Lemmas 20-22]{LP20on}}$]\label{l:rank2system}
	Let $\lambda_1$ and $\lambda_2$ be two nonzero constants. Suppose $a\in C^2(\R)$ satisfies $a'>0$ and $a''>0$ on $\R$, and $F\in C^3(\R)$ satisfies $F' = a$. In addition, assume $a(0) = F(0) = 0$. Then:
	
	\begin{enumerate}[(a)]
		\item The system \eqref{eq:rank2system} has at most two distinct solutions satisfying $a(t)<\lambda_1$.
		\item Let $(s, t)$ be a nontrivial solution of the system \eqref{eq:rank2system}, i.e., $(s,t)\ne (0,0)$, with $s^2-ta(t)\neq0$. If $\lambda_1>0$ and $a(t)<\lambda_1$, then $s^2-ta(t)>0$; on the other hand, if $\lambda_1<0$ and $a(t)>\lambda_1$, then $s^2-ta(t)<0$.
		\item Suppose $\lambda_1>0$. If $(s_1,t_1)$ and $(s_2,t_2)$ are two nontrivial solutions of the system \eqref{eq:rank2system} with $\lambda_1<a(t_1)<a(t_2)$, then $s_1^2-t_1a(t_1)>s_2^2-t_2a(t_2)$. 
	\end{enumerate}
\end{lem}

In Lemma \ref{l:example} in Appendix~\ref{sec:example}, we provide an explicit example of a system in the form of \eqref{eq:rank2system} that has at least 6 distinct solutions.

\section{Proof of Theorem \ref{thm_tn}}\label{sec:main}

Recalling the notation from \eqref{eq_atk}-\eqref{eq_Pi}, for every $i=1, \dots, N$, the matrix $A_{\mathcal{X}} - \Pi^N(X_i)$ is obtained by subtracting the $i$-th column and the $(i+N)$-th column of the matrix $A_{\mathcal{X}} - \Pi^N(Q)$ from its first $N$ columns and last $N$ columns, respectively. Equivalently, $A_{\mathcal{X}} - \Pi^N(X_i)$ is obtained by performing elementary column operations on $A_{\mathcal{X}} - \Pi^N(Q)$ and setting the $i$-th and $(i+N)$-th columns of $A_{\mathcal{X}} - \Pi^N(Q)$ to zero. Hence, the assumption \eqref{eq:r2assumption} implies
\begin{equation}\label{eq:r2i}
	 \mathrm{rank}\lt(A_{\mathcal{X}} - \Pi^N(X_i)\rt)\leq 2\qquad\forall i=1,\dots, N.
\end{equation}
Recalling the definitions for $h_j^i$, $r_j^i, a^i$ and $F^i$ in \eqref{eq:hiri}-\eqref{eq:aiFi}, it can be checked directly that
\begin{align}\label{eq:ha}
	h_j^i a^i(r_j^i) &= \lt(u_ja(v_j)-u_i a(v_i)\rt) - a(v_i)\lt(u_j-u_i\rt) - u_i\lt(a(v_j)-a(v_i)\rt)\nn\\
	&=\lt(u_ja(v_j)-u_i a(v_i)\rt) -a(v_i)h_j^i - u_i a^i(r_j^i),
\end{align}
\begin{align}\label{eq:hF}
	\frac{(h_j^i)^2}{2}+F^i(r_j^i) &= \lt(\frac{u_j^2}{2}+F(v_j)-\frac{u_i^2}{2}-F(v_i)\rt) - a(v_i)\lt(v_j-v_i\rt) - u_i\lt(u_j-u_i\rt)\nn\\
	&= \lt(\frac{u_j^2}{2}+F(v_j)-\frac{u_i^2}{2}-F(v_i)\rt) - a(v_i)r_j^i - u_i h_j^i.
\end{align}
Therefore, performing elementary row operations on $A_{\mathcal{X}} - \Pi^N(X_i)$ leads to the matrix
 \begin{equation*}
	A_{\mathcal{X}}^i:=\begin{bmatrix}
		h_1^i&\null&h_{N}^i&r_1^i&\null&r_{N}^i\\a^i(r_1^i)&\cdots&a^i(r_{N}^i)&h_1^i&\cdots&h_{N}^i\\h_1^ia^i(r_1^i)&\null&h_{N}^ia^i(r_{N}^i)&\frac{(h_1^i)^2}{2}+F^i(r_1^i)&\null&\frac{(h_{N}^i)^2}{2}+F^i(r_{N}^i)
	\end{bmatrix},
\end{equation*}
and \eqref{eq:r2i} implies 
\begin{equation*}
	\mathrm{rank}\lt(A_{\mathcal{X}}^i\rt)\leq 2\qquad\forall i=1,\dots, N.
\end{equation*}

In view of Lemma \ref{l:TXi} and Remark \ref{r:linindependent}, we may assume that $\mathrm{rank}\lt(A_{\mathcal{X}}^i\rt)=2$ for every $i=1,\dots, N$, and there exist constants $\lambda_1^i, \lambda_2^i$ such that
	\begin{equation}\label{eq:rank2system_i'}
		\begin{cases}
			&h_j^i a^i(r_j^i) = \lambda_1^i h_j^i + \lambda_2^i a^i(r_j^i),\\
			&\frac{(h_j^i)^2}{2}+ F^i(r_j^i) = \lambda_1^i r_j^i + \lambda_2^i h_j^i,
		\end{cases}\qquad\text{ for all }j=1, \dots, N.
	\end{equation} 
Now Theorem \ref{thm_tn} follows directly from the following lemma.

\begin{lem}\label{l:tn}
	Suppose $a\in C^2(\R)$ satisfies $a'>0$ and $a''>0$ on $\R$. Let $N\geq 6$ and $\mathcal{X}=\{X_i\}_{i=1}^N\subset K_a$ with $X_i = P(u_i, v_i)$. Assume in addition that for every $i=1, \dots, N$, there exist constants $\lambda_1^i, \lambda_2^i$ such that the system \eqref{eq:rank2system_i'} holds. Then $\mathcal{X}$ does not have an order that forms a $T_N$ configuration.
\end{lem}

\begin{proof}
	Recall the properties of $a^i$ and $F^i$ in \eqref{eq:aiproperty}-\eqref{eq:Fiproperty}, and thus Lemma \ref{l:degenerate} applies to the systems \eqref{eq:rank2system_i'}. In particular, we may assume in addition that
	\begin{equation*}
		\lambda_1^i\ne 0, \qd\lambda_2^i\ne 0\qquad\text{ for all }i=1,\dots, N.
	\end{equation*}
	This makes Lemma \ref{l:rank2system} applicable to \eqref{eq:rank2system_i'}. Note that for each $i$, we must have
	\begin{equation}\label{eq:nondegenerate}
		a^i(r_j^i) \ne \lambda_1^i\qquad\text{ for all }j=1,\dots, N.
	\end{equation}
	Otherwise, the first equation in \eqref{eq:rank2system_i'} would reduce to $0=\lambda_2^i a^i(r_j^i) = \lambda_2^i \lambda_1^i$, which is a contradiction.
	
	Without loss of generality, we assume $v_1\leq\dots\leq v_N$. If there exists $i$ such that $v_i = v_{i+1}$, then $r_{i+1}^i = a^i(r_{i+1}^i) = 0$ by \eqref{eq:hiri} and \eqref{eq:aiproperty}. It follows from the first equation in \eqref{eq:rank2system_i'} with $j=i+1$ that $\lambda_1^i h_{i+1}^i=0$, and thus $h_{i+1}^i=0$ as $\lambda_1^i\ne 0$. By definition \eqref{eq:hiri}, this means $u_i=u_{i+1}$, and therefore $X_i = X_{i+1}$ and $\mathrm{Card}(\mathcal{X})<N$. So in the following we assume
	\begin{equation*}
		v_1<v_2<\dots< v_N.
	\end{equation*}
	Since $r_j^i=v_j-v_i$, this order implies that
	\begin{equation}\label{eq:rijsign}
		r_j^i
		\begin{cases}
			&>0\qd\text{ if }j>i,\\
			&<0\qd\text{ if }j<i.
		\end{cases}
	\end{equation}
	And it is clear that $r_j^i=-r_i^j$ and $h_j^i=-h_i^j$ for all $i,j\in\{1,\dots,N\}$. We denote (recalling the notation from Proposition \ref{def_tn_prop} and \eqref{eq:TXi})
	\begin{equation}
		\label{eq_2Dsysdet}
		D_j^i:=\det\lt((T_j^i)^{12} - (T_i^i)^{12}\rt) =\det\lt((T_j^i)^{12}\rt) = (h_j^i)^2-r_j^ia^i(r_j^i).
	\end{equation}
	Further, note that (recalling \eqref{eq:aiFi})
	\begin{align*}
		a^i(r_j^i)&=a(v_i+r_j^i)-a(v_i)=a(v_j)-a(v_i)\\
		&=-(a(v_i)-a(v_j))=-(a(v_j+r_i^j)-a(v_j))=-a^j(r_i^j),
	\end{align*}
	and therefore it holds 
	\begin{equation*}
		D_j^i=D_i^j. 
	\end{equation*}
	Lastly, note that if $D_j^i = 0$ for some $j\ne i$, then $\mathrm{rank}\lt((T_j^i)^{12}\rt) = 1$. The system \eqref{eq:rank2system_i'} says that the last row of $T_j^i$ is a linearly combination of its first two rows, and thus $\mathrm{rank}(T_j^i) = \mathrm{rank}(T_j^i-T_i^i) = 1$, i.e., the set $\mathbb{T}_{\mathcal{X}}^i$ contains rank-one connections and thus does not have an order that forms a $T_N$ configuration. By Lemma \ref{l:TXi}, the same holds true for $\mathcal{X}$. So in the following we further assume
	\begin{equation*}
		D_j^i \ne 0\qquad\text{ for all } j\ne i.
	\end{equation*}
	
	\begin{table}[ht]
		\begin{minipage}{0.45\textwidth}
			\begin{center}
				\begin{tabular}{c|ccccc|}
					$D_j^i$ & j=1&2&3&$\cdots$&$N$ \\\hline
					i=1&$0$&$?$&$?$&$\cdots$&$?$\\
					2&$?$&$0$&$?$&$\cdots$&$?$\\
					3&$?$&$?$&0&$\cdots$&$?$\\
					$\vdots$&$\vdots$&$\vdots$&$\vdots$&$\ddots$&$\vdots$\\
					$N$&$?$&$?$&$?$&$\cdots$&$0$\\\hline
				\end{tabular}
			\end{center}
			\caption{Symmetric table for signs of $D_j^i$.}
			\label{t:1}
		\end{minipage}\hfill
		\begin{minipage}{0.45\textwidth}
			\begin{center}
				\begin{tabular}{c|ccccc|}
					$D_j^i$ & j=1&2&3&$\cdots$&$N$ \\\hline
					i=1&$0$&$-$&$-$&$\cdots$&$-$\\
					2&$-$&\null&\null&$\cdots$&\null\\
					3&$-$&\null&\null&$\cdots$&\null\\
					$\vdots$&$\vdots$&$\vdots$&$\vdots$&$\ddots$&$\vdots$\\
					$N$&$-$&\null&\null&$\cdots$&\null\\\hline
				\end{tabular}
			\end{center}
			\caption{Case 1.}
			\label{t:c1}
		\end{minipage}
	\end{table}
	
	In the remaining proof, we consider the symmetric Table \ref{t:1} and will show that there exists $i\in\{1,\dots,N\}$ such that the $i$-th row (or equivalently the $i$-th column) of the table does not change sign. This implies via \eqref{eq_2Dsysdet} that the set $\lt\{\det\lt((T_j^i)^{12} - (T_i^i)^{12}\rt)\rt\}_{j\ne i}$ does not change sign. By Proposition \ref{def_tn_prop}, this further implies that the set $\mathbb{T}_{\mathcal{X}}^i$ does not have an order that forms a $T_N$ configuration, and hence the same holds true for $\mathcal{X}$ as a consequence of Lemma \ref{l:TXi}.
	\medskip

	\emph{Case 1.} Suppose $\lambda_1^1<0$. Then, recalling \eqref{eq:rijsign}, for all $1< j\leq N$, we have $r_j^1>0$. So $a^1(r_j^1)>0>\lambda_1^1$, and Lemma \ref{l:rank2system}$(b)$ implies that $D_j^1<0$ for all $1<j\leq N$. Hence the first row of Table \ref{t:1} does not change sign; see Table \ref{t:c1}.
	\medskip
	
	\emph{Case 2.} Suppose $\lambda_1^1>0$ and $\lambda_1^2<0$. 
	\medskip
	
	\emph{Step 1.} By \eqref{eq:rijsign}, we have $r_1^2<0$, and  $r_j^2>0$ for all $j>2$. Therefore, recalling \eqref{eq:aiproperty}, we have $a^2(r_j^2)>0>\lambda_1^2$ for all $j>2$. Again Lemma \ref{l:rank2system}$(b)$ implies that $D_j^2<0$ for all $j>2$; see Table \ref{t:c2.1}.
	
	\begin{table}[ht]
		\begin{minipage}{0.45\textwidth}
			\begin{center}
				\begin{tabular}{c|ccccc|}
					$D_j^i$ & j=1&2&3&$\cdots$&$N$ \\\hline
					i=1&0&$?$&\null&$\cdots$&\null\\
					2&$?$&$0$&$-$&$\cdots$&$-$\\
					3&\null&$-$&\null&$\cdots$&\null\\
					$\vdots$&$\vdots$&$\vdots$&$\vdots$&$\ddots$&$\vdots$\\
					$N$&\null&$-$&\null&$\cdots$&\null\\\hline
				\end{tabular}
			\end{center}
			\caption{Case 2, Step 1.}
			\label{t:c2.1}
		\end{minipage}\hfill
		\begin{minipage}{0.45\textwidth}
			\begin{center}
				\begin{tabular}{c|ccccc|}
					$D_j^i$ & j=1&2&3&$\cdots$&$N$ \\\hline
					i=1&0&$?$&\null&$\cdots$&$?$\\
					2&$?$&$0$&$-$&$\cdots$&$-$\\
					3&\null&$-$&\null&$\cdots$&$-$\\
					$\vdots$&$\vdots$&$\vdots$&$\vdots$&$\ddots$&$\vdots$\\
					$N$&$?$&$-$&$-$&$\cdots$&$0$\\\hline
				\end{tabular}
			\end{center}
			\caption{Case 2, Step 2.}
			\label{t:c2.2}
		\end{minipage}
	\end{table}

	\emph{Step 2.} Note that $r_1^N<r_2^N<\dots<r_{N-1}^N<0$, and thus $a^N(r_1^N)<a^N(r_2^N)<\dots<a^N(r_{N-1}^N)<0$ by \eqref{eq:aiproperty}. Then Lemma \ref{l:rank2system}$(a)$ implies $\lambda_1^N<0$, and at most two of $a^N(r_j^N)$ satisfy $a^N(r_j^N)<\lambda_1^N$. If $\lambda_1^N<a^N(r_1^N)$, then Lemma \ref{l:rank2system}$(b)$ implies that $D_j^N<0$ for all $1\leq j<N$, and the last row of Table \ref{t:1} does not change sign. Otherwise, Lemma \ref{l:rank2system}$(b)$ implies that at most two elements in the last row of Table \ref{t:1}, namely $D^N_1$ and/or $D^N_2$, can be positive, and the rest are negative. We already know $D^N_2 = D^2_N<0$, so Table \ref{t:c2.1} becomes Table \ref{t:c2.2}.
	
	\emph{Step 3.} If $D^N_1<0$, then we are done. Suppose $D^N_1>0$. By Lemma \ref{l:rank2system}$(a)$, since $a^1(r^1_1)=0<\lambda_1^1$, we must have $0<\lambda_1^1<a^1(r_3^1)<\dots<a^1(r_N^1)$, and it follows from Lemma \ref{l:rank2system}$(c)$ that $D_3^1>\dots>D_N^1=D^N_1>0$. So Table \ref{t:c2.2} becomes Table \ref{t:c2.3}. Now if $D_1^2 = D_2^1>0$, then the first row of Table \ref{t:c2.3} does not change sign; if $D_1^2 = D_2^1<0$, then the second row of Table \ref{t:c2.3} does not change sign.
	
	\begin{table}[ht]
		\begin{center}
			\begin{tabular}{c|ccccc|}
				$D_j^i$ & j=1&2&3&$\cdots$&$N$ \\\hline
				i=1&$0$&$?$&$+$&$\cdots$&$+$\\
				2&$?$&$0$&$-$&$\cdots$&$-$\\
				3&$+$&$-$&\null&$\cdots$&$-$\\
				$\vdots$&$\vdots$&$\vdots$&$\vdots$&$\ddots$&$\vdots$\\
				$N$&$+$&$-$&$-$&$\cdots$&$0$\\\hline
			\end{tabular}
		\end{center}
		\caption{Case 2, Step 3.}
		\label{t:c2.3}
	\end{table}
	\medskip
	
	\emph{Case 3.} Suppose $\lambda_1^1>0$ and $\lambda_1^2>0$. 
	\medskip
	
	\emph{Step 1.} Consider again the last row. As in Case 2, Step 2, we know that at most two elements in the last row of Table \ref{t:1}, namely $D^N_1$ and/or $D^N_2$, can be positive, and the rest are negative. If $D_j^N<0$ for all $1\leq j<N$, then we are done. Otherwise, first assume $D^N_1>0$. As in Case 2, Step 3, by Lemma \ref{l:rank2system}$(a)$, we must have $0<\lambda_1^1<a^1(r_3^1)<\dots<a^1(r_N^1)$, and Lemma \ref{l:rank2system}$(c)$ implies that $D_3^1>\dots>D_N^1=D^N_1>0$. Since $\lambda_1^2>0>a^2(r_1^2)$, by Lemma \ref{l:rank2system}$(b)$, we also have $D_1^2=D_2^1>0$, so the first row of Table \ref{t:1} does not change sign; see Table \ref{t:c3.1}.
	
	\begin{table}[ht]
		\begin{minipage}{0.45\textwidth}
			\begin{center}
				\begin{tabular}{c|ccccc|}
					$D_j^i$ & j=1&2&3&$\cdots$&$N$ \\\hline
					i=1&0&$+$&+&$\cdots$&$+$\\
					2&$+$&0&\null&$\cdots$&$?$\\
					3&+&\null&\null&$\cdots$&$-$\\
					$\vdots$&$\vdots$&$\vdots$&$\vdots$&$\ddots$&$\vdots$\\
					$N$&$+$&$?$&$-$&$\cdots$&$0$\\\hline
				\end{tabular}
			\end{center}
			\caption{Case 3, Step 1.}
			\label{t:c3.1}
		\end{minipage}\hfill
		\begin{minipage}{0.45\textwidth}
			\begin{center}
				\begin{tabular}{c|ccccc|}
					$D_j^i$ & j=1&2&3&$\cdots$&$N$ \\\hline
					i=1&$0$&$+$&$?$&$\cdots$&$?$\\
					2&$+$&$0$&$+$&$\cdots$&$+$\\
					3&$?$&$+$&\null&$\cdots$&$-$\\
					$\vdots$&$\vdots$&$\vdots$&$\vdots$&$\ddots$&$\vdots$\\
					$N$&$?$&$+$&$-$&$\cdots$&$0$\\\hline
				\end{tabular}
			\end{center}
			\caption{Case 3, Step 2.}
			\label{t:c3.2}
		\end{minipage}
	\end{table}
	
	\emph{Step 2.} It remains to consider the case $D^N_2>0$. Since $a^2(r_1^2)<a^2(r_2^2)=0<a^2(r_3^2)<\dots<a^2(r_N^2)$, Lemma \ref{l:rank2system}$(a)$ implies that $0<\lambda_1^2<a^2(r_3^2)<\dots<a^2(r_N^2)$. Once again Lemma \ref{l:rank2system}$(c)$ implies that $D_3^2>\dots>D_N^2=D^N_2>0$. As in Step 1, since $\lambda_1^2>0>a^2(r_1^2)$, by Lemma \ref{l:rank2system}$(b)$, we also have $D_1^2=D_2^1>0$, so the second row of Table \ref{t:1} does not change sign; see Table \ref{t:c3.2}.
\end{proof}

\section{Proof of Corollary \ref{thm_general2Dtn}}\label{sec:proof}

Suppose $\mathcal{X} = \{X_1,\dots,X_N\}\subset K_a\cap \mathcal{P}$ for some two-dimensional plane $\mathcal{P}\subset\R^{3\times 2}$. If there exists $i$ such that the dimension of $\mathrm{Span}\{X_1-X_i,\dots,X_N-X_i\}$ is $1$, then $\mathcal{X}$ cannot have an order that forms a $T_N$ configuration, for otherwise the rank-one matrices $C_i$ in Definition \ref{def_tn} would satisfy $C_i\in \mathrm{Span}\{X_1-X_i,\dots,X_N-X_i\}$ (see \cite[Eq. (20)]{LP20on}) and thus all the points in $\mathcal{X}$ would be rank-one connected. So in the following we assume that
\begin{equation}\label{eq:dim2}
	\mathrm{dim}\lt(\mathrm{Span}\{X_1-X_i,\dots,X_N-X_i\}\rt) = 2\qquad\text{ for all }i=1,\dots, N.
\end{equation}
Corollary \ref{thm_general2Dtn} follows directly from the following Lemma \ref{l:2Dsysi} and Lemma \ref{l:tn}.

\begin{lem}\label{l:2Dsysi}
	Suppose $a\in C^2(\R)$ satisfies $a'>0$ and $a''>0$ on $\R$. Let $N\geq 6$ and $\mathcal{X}=\{X_i\}_{i=1}^N\subset K_a$ with $X_i = P(u_i, v_i)$. Assume in addition that $\mathcal{X}$ satisfies \eqref{eq:dim2}. If $\mathcal{X}$ has an order that forms a $T_N$ configuration, then, recalling \eqref{eq:hiri}-\eqref{eq:aiFi}, for each $i=1, \dots, N$, there exist constants $\alpha_1^i, \alpha_2^i, \lambda_1^i, \lambda_2^i$ such that
\begin{equation}\label{eq:rank2system_i}
	\begin{cases}
		&a^i(r_j^i)=\alpha_1^ih_j^i+\alpha_2^ir_j^i,\\
		&h_j^i a^i(r_j^i) = \lambda_1^i h_j^i + \lambda_2^i a^i(r_j^i),\\
		&\frac{(h_j^i)^2}{2}+ F^i(r_j^i) = \lambda_1^i r_j^i + \lambda_2^i h_j^i,
	\end{cases}\qquad\text{ for all }j=1, \dots, N,
\end{equation} 
or 
\begin{equation}\label{eq:rank2system_i2}
	\begin{cases}
		&r_j^i=\alpha_1^ih_j^i+\alpha_2^ia^i(r_j^i),\\
		&h_j^i a^i(r_j^i) = \lambda_1^i h_j^i + \lambda_2^i a^i(r_j^i),\\
		&\frac{(h_j^i)^2}{2}+ F^i(r_j^i) = \lambda_1^i r_j^i + \lambda_2^i h_j^i,
	\end{cases}\qquad\text{ for all }j=1, \dots, N.
\end{equation} 
\end{lem}

\begin{proof}
Denote the matrix $S_{\mathcal{X}}^i\in\R^{5\times N}$ by
\begin{equation*}
	S_{\mathcal{X}}^i:=\begin{bmatrix}
		h_1^i&\null&h_N^i\\r_1^i&\null& r_N^i\\
		a^i(r_1^i)&\cdots&a^i(r_N^i)\\h_1^ia^i(r_1^i)&\null&h_N^ia^i(r_N^i)\\
		\frac{(h_1^i)^2}{2}+F^i(r_1^i)&\null&\frac{(h_N^i)^2}{2}+F^i(r_N^i)
	\end{bmatrix}.
\end{equation*}
It follows from \eqref{eq:ha}-\eqref{eq:hF} that the assumption \eqref{eq:dim2} is equivalent to 
\begin{equation*}
	\mathrm{rank}(S_{\mathcal{X}}^i) = 2\qquad\text{ for all }i=1,\dots, N.
\end{equation*}
Fix some $i\in\{1,\dots,N\}$. By Lemmas \ref{l:linindependent}-\ref{l:TXi} and Remark \ref{r:linindependent}, we infer that the first and second rows of $S_{\mathcal{X}}^i$ are linearly independent, or the first and third rows of $S_{\mathcal{X}}^i$ are linearly independent.
\medskip

\emph{Case 1.}  If the first and second rows of $S_{\mathcal{X}}^i$ are linearly independent, then there exist $\alpha_1^i, \alpha_2^i, \beta_1^i, \beta_2^i, \gamma_1^i, \gamma_2^i \in\R$ such that, for all $1\leq j\leq N$,
\begin{equation}
	\label{eq_2Dsysrow1}
	a^i(r_j^i)=\alpha_1^ih_j^i+\alpha_2^ir_j^i,
\end{equation}
\begin{equation}
	\label{eq_2Dsysrow2}
	h_j^ia^i(r_j^i)=\beta_1^ih_j^i+\beta_2^ir_j^i,
\end{equation}
and
\begin{equation}
	\label{eq_2Dsysrow3}
	\frac{(h_j^i)^2}{2}+F^i(r_j^i)=\gamma_1^ih_j^i+\gamma_2^ir_j^i.
\end{equation}
Recalling \eqref{eq:TXi}, this implies that
\begin{equation*}
\mathrm{Span}\{\mathbb{T}_{\mathcal{X}}^i\}=\left\{\begin{bmatrix}
	s&t\\\alpha_1^is+\alpha_2^it&s\\\beta_1^is+\beta_2^it&\gamma_1^is+\gamma_2^it
\end{bmatrix}:s,t\in\mathbb{R}\right\}.
\end{equation*}
For any $M\in \mathrm{Span}\{\mathbb{T}_{\mathcal{X}}^i\}$, recalling the notation from Proposition \ref{def_tn_prop}, we have
\begin{equation*}
	\det\lt(M^{12}\rt)=s^2-\alpha_1^i st-\alpha_2^i t^2
\end{equation*}
and
\begin{equation*}
	\det\lt(M^{13}\rt)=\gamma_1^i s^2 + (\gamma_2^i-\beta_1^i)st - \beta_2^i t^2. 
\end{equation*}
If $\mathcal{X}$ has an order that forms a $T_N$ configuration, then $\mathbb{T}_{\mathcal{X}}^i$ also has an order that forms a $T_N$ configuration by Lemma \ref{l:TXi}. This requires the two-dimensional subspace $\mathrm{Span}\{\mathbb{T}_{\mathcal{X}}^i\}\subset\R^{3\times 2}$ to contain at least two distinct rank-one directions, for otherwise all the points in $\mathbb{T}_{\mathcal{X}}^i$ would be rank-one connected. Equivalently, this requires the above two functions, viewed as quadratic functions of $s$, to share at least two distinct solutions. Therefore, we deduce that
\begin{equation*}
	\begin{cases}
		\gamma_2^i-\beta_1^i = -\gamma_1^i \alpha_1^i,\\
		\beta_2^i = \gamma_1^i \alpha_2^i.
	\end{cases}
\end{equation*}
Solving for $\beta_1^i, \beta_2^i$ from the above system and plugging the expressions into \eqref{eq_2Dsysrow2}, we obtain
\begin{equation*}
	h_j^ia^i(r_j^i)=(\gamma_2^i+\gamma_1^i\alpha_1^i)h_j^i+\gamma_1^i\alpha_2^i r_j^i. 
\end{equation*}
Further, from \eqref{eq_2Dsysrow1}, we have $\alpha_2^ir_j^i = a^i(r_j^i) - \alpha_1^i h_j^i$. Plugging this into the above and simplifying the equation leads to
\begin{equation*}
	h_j^ia^i(r_j^i) = \gamma_2^i h_j^i + \gamma_1^i  a^i(r_j^i). 
\end{equation*}
This together with \eqref{eq_2Dsysrow1} and \eqref{eq_2Dsysrow3} gives a system in the form of \eqref{eq:rank2system_i} as desired.
\medskip

\emph{Case 2.}  If the first and third rows of $S_{\mathcal{X}}^i$ are linearly independent, then there exist $\alpha_1^i, \alpha_2^i, \beta_1^i, \beta_2^i, \gamma_1^i, \gamma_2^i \in\R$ such that, for all $1\leq j\leq N$,
\begin{equation}
	\label{eq_2Dsysrow1.2}
	r_j^i=\alpha_1^ih_j^i+\alpha_2^ia^i(r_j^i),
\end{equation}
\begin{equation}
	\label{eq_2Dsysrow2.2}
	h_j^ia^i(r_j^i)=\beta_1^ih_j^i+\beta_2^ia^i(r_j^i),
\end{equation}
and
\begin{equation}
	\label{eq_2Dsysrow3.2}
	\frac{(h_j^i)^2}{2}+F^i(r_j^i)=\gamma_1^ih_j^i+\gamma_2^ia^i(r_j^i).
\end{equation}
Then we have
\begin{equation*}
\mathrm{Span}\{\mathbb{T}_{\mathcal{X}}^i\}=\left\{\begin{bmatrix}
	s&\alpha_1^is+\alpha_2^it\\t&s\\\beta_1^is+\beta_2^it&\gamma_1^is+\gamma_2^it
\end{bmatrix}:s,t\in\mathbb{R}\right\}.
\end{equation*}
Again we need the two-dimensional subspace $\mathrm{Span}\{\mathbb{T}_{\mathcal{X}}^i\}\subset\R^{3\times 2}$ to contain at least two distinct rank-one directions. Using $\det\lt(M^{12}\rt)$ and $\det\lt(M^{23}\rt)$ for $M\in \mathrm{Span}\{\mathbb{T}_{\mathcal{X}}^i\}$ along the same lines as in Case 1, we deduce that
\begin{equation*}
	\begin{cases}
		\beta_2^i-\gamma_1^i = -\beta_1^i \alpha_1^i,\\
		\gamma_2^i = \beta_1^i \alpha_2^i.
	\end{cases}
\end{equation*}
Solving for $\gamma_1^i, \gamma_2^i$ from the above system and plugging the expressions into \eqref{eq_2Dsysrow3.2}, we obtain
\begin{equation*}
	\frac{(h_j^i)^2}{2}+F^i(r_j^i)=(\beta_2^i+\beta_1^i\alpha_1^i)h_j^i+\beta_1^i\alpha_2^i a^i(r_j^i). 
\end{equation*}
Further, using \eqref{eq_2Dsysrow1.2} to replace $\alpha_2^ia^i(r_j^i)$ in the above second term on the right hand side, we obtain
\begin{equation*}
	\frac{(h_j^i)^2}{2}+F^i(r_j^i) = \beta_1^i r_j^i + \beta_2^i h_j^i. 
\end{equation*}
This together with \eqref{eq_2Dsysrow1.2} and \eqref{eq_2Dsysrow2.2} gives a system in the form of \eqref{eq:rank2system_i2} as desired.
\end{proof}

\begin{rem}
	It could be that the systems \eqref{eq:rank2system_i} and \eqref{eq:rank2system_i2} are overdetermined and therefore lack sufficiently many solutions. A more careful analysis of these two systems may lead to a proof of Corollary \ref{thm_general2Dtn} without relying on Lemma \ref{l:tn}.
\end{rem}

\begin{rem}\label{r:rank3}
	With only the two equations in the system \eqref{eq:rank2system_i'}, the matrix $S_{\mathcal{X}}^i$ may have rank 3, corresponding to the case
	\begin{equation*}
		\mathrm{dim}\lt(\mathrm{Span}\{X_1-X_i,\dots,X_N-X_i\}\rt) = 3.
	\end{equation*}
\end{rem}

\begin{appendices}
	\section{An example satisfying the assumptions in Theorem \ref{thm_tn}}\label{sec:example}

\begin{lem}\label{l:example}
	There exist a function $a\in C^2(\R)$ with $a'>0$ and $a''>0$ on $\R$, and a set $\mathcal{X}\subset K_a$ consisting of at least $6$ points, such that the condition \eqref{eq:r2assumption} is satisfied.
\end{lem}

\begin{proof}
	Let $k>0$ be a constant that will be specified later. Consider the function $a:\R\to\R$ given by
	\begin{equation}\label{eq:a}
		a(t):=\begin{cases}
			e^t-1 &\text{if }t\leq 0,\\
			\frac{k}{6}t^3 + \frac 12 t^2 + t &\text{if }t>0. 
		\end{cases}
	\end{equation}
	Then 
	\begin{equation*}
		a'(t)=\begin{cases}
			e^t &\text{if }t\leq 0,\\
			\frac{k}{2}t^2 + t + 1 &\text{if }t>0,
		\end{cases}
	\end{equation*}
	and
	\begin{equation}\label{eq:a''}
		a''(t)=\begin{cases}
			e^t &\text{if }t\leq 0,\\
			kt + 1 &\text{if }t>0. 
		\end{cases}
	\end{equation}
	Therefore, $a\in C^2(\R)$ satisfies $a'>0$ and $a''>0$ on $\R$. Further, define 
		\begin{equation*}
		F(t):=\begin{cases}
			e^t-t - 1 &\text{if }t\leq 0,\\
			\frac{k}{24}t^4 + \frac 16 t^3 + \frac 12 t^2 &\text{if }t>0,
		\end{cases}
	\end{equation*}
	so that $F' = a$, and $a(0)=F(0)=0$. 
	
	Now consider the system of equations \eqref{eq:rank2system}, which we rewrite here in the unknowns $u, v$:
	\begin{equation}\label{eq:r2system}
		\begin{cases}
			&u a(v) = \lambda_1 u+ \lambda_2 a(v),\\
			&\frac{u^2}{2}+ F(v) = \lambda_1 v + \lambda_2 u.
		\end{cases}
	\end{equation} 
	We claim that, for an appropriate choice of the constants $k>0$ and $\lambda_1\ne 0, \lambda_2\ne 0$, the above system has at least $6$ distinct solutions $(u_i, v_i)$ for $i=1,\dots, 6$, including the trivial solution $(0,0)$. Taking this for granted, define the set $\mathcal{X}:=\{P(u_i, v_i)\}_{i=1}^6 \subset K_a$, and let $Q=P(0,0) = 0\in \R^{3\times 2}$. Recalling the definitions of $A_{\mathcal{X}}$ and $\Pi^N(Q)$ in \eqref{eq_atk} and \eqref{eq_Pi}, respectively, $(u_i, v_i)$ being solutions of the system \eqref{eq:r2system} implies $\mathrm{rank}\lt(A_{\mathcal{X}} - \Pi^6(Q)\rt)\leq 2$. If $\mathrm{rank}\lt(A_{\mathcal{X}} - \Pi^6(Q)\rt)=1$, then, recalling \eqref{eq:uva}, the vectors $\vec u$ and $\vec v$ would be linearly dependent, so would $\vec u$ and $\vec {\mathfrak{a}}$. The proof of Lemma \ref{l:linindependent} would then imply $\mathrm{Card}(\mathcal{X})\leq 2$, which is a contradiction. Hence, the condition \eqref{eq:r2assumption} is satisfied.
	
	It remains to choose $k>0$ and $\lambda_1\ne 0, \lambda_2\ne 0$ so that the system \eqref{eq:r2system} has at least 6 distinct solutions. To this end, we first borrow some calculations from the proof of \cite[Lemma 20]{LP20on}. The argument leading to \eqref{eq:nondegenerate} shows that any solution $(u,v)$ of \eqref{eq:r2system} must satisfy $a(v)\ne \lambda_1$, provided $\lambda_1\ne 0, \lambda_2\ne 0$. Solving for $u$ from the first equation in \eqref{eq:r2system}, we obtain 
	\begin{equation}\label{eq:u}
		u = \frac{\lambda_2 a(v)}{a(v) - \lambda_1}. 
	\end{equation}
	Plugging this into the second equation in \eqref{eq:r2system} and simplifying it gives
	\begin{equation}
		\label{eqp33}
		F(v)-\lm_1 v-\frac{\lm_2^2}{2} = -\frac{\lm_1^2\lm_2^2}{2(a(v)-\lm_1)^2}.
	\end{equation}
	Let us denote
	\begin{equation*}
		p(v):=F(v)-\lm_1 v-\frac{\lm_2^2}{2},\qquad q(v):=-\frac{\lm_1^2\lm_2^2}{2(a(v)-\lm_1)^2}.
	\end{equation*}
	Direct calculations using $F'=a$ show that
	\begin{equation}\label{eq:p}
		p'(v)=a(v)-\lm_1, \qquad p''(v)=a'(v),
	\end{equation}
	and
	\begin{equation}\label{eq:q}
		q'(v)=\frac{\lm_1^2\lm_2^2a'(v)}{(a(v)-\lm_1)^3}, \qquad q''(v)=\lm_1^2\lm_2^2\frac{a''(v)\lt(a(v)-\lm_1\rt)-3a'(v)^2}{(a(v)-\lm_1)^{4}}.
	\end{equation}
	Further, note that $a(0)=F(0)=0$, and $a'(0) = a''(0)=1$, so that
	\begin{equation}\label{eq:at0}
		p'(0) = -\lambda_1, \qquad p''(0) = 1, \qquad q'(0) = -\frac{\lambda_2^2}{\lambda_1}, \qquad q''(0) = \frac{\lambda_2^2}{\lambda_1^2}\lt(-\lambda_1-3\rt). 
	\end{equation}
	Now we choose
	\begin{equation}\label{eq:lambda}
		\lambda_1 = -\frac 12, \qquad \lambda_2 = \frac{1}{10}. 
	\end{equation}
	In the following, we show that equation \eqref{eqp33} has at least two distinct solutions in each of the intervals $(-\infty, \ln(1/2))$, $(\ln(1/2),0]$, and $(0,\infty)$, and each solution provides a solution $(u,v)$ of \eqref{eq:r2system} via \eqref{eq:u}.
	\medskip
	
	\emph{Case 1.} Consider $v<\ln(1/2)$, i.e., $a(v)<\lambda_1$. The calculations in \eqref{eq:p}-\eqref{eq:q} show that $p''(v) = a'(v)>0$, and $q''(v)<0$ for $a(v)<\lambda_1$ as $a''(v)>0$. Thus \eqref{eqp33} can have up to 2 distinct solutions for $v<\ln(1/2)$. Explicit calculations, e.g., using MATLAB, show that
	\begin{equation*}
		p(-1) < q(-1). 
	\end{equation*}
	Further, from the explicit formulas for $p$ and $q$, it is clear that $p(v)\to \infty$ as $v\to -\infty$, and $q(v) \to -1/200$ as $v\to -\infty$ and $q(v)\to -\infty$ as $v\to \ln(1/2)$. Thus \eqref{eqp33} has two distinct solutions for $v<\ln(1/2)$.
	\medskip
	
	\emph{Case 2.} Consider $\ln(1/2)<v\leq 0$, i.e., $\lambda_1<a(v)\leq 0$. Note that $v=0$ is a trivial solution of \eqref{eqp33}. Using MATLAB, we find
	\begin{equation*}
		p(-0.3) < q(-0.3). 
	\end{equation*}
	Again, as $q(v)\to -\infty$ as $v\to \ln(1/2)$, equation \eqref{eqp33} has at least one solution in $(\ln(1/2), -0.3)$. Together with 0, equation \eqref{eqp33} has at least two distinct solutions in $(\ln(1/2), 0]$.
	\medskip
	
	\emph{Case 3.} Finally, consider $v>0$. From \eqref{eq:at0}-\eqref{eq:lambda}, we deduce 
	\begin{equation}\label{eq:at0'}
		p'(0) > q'(0), \qquad p''(0) > 0 > q''(0). 
	\end{equation}
	Here the idea is that, recalling \eqref{eq:a''}, by choosing $k>0$ sufficiently large, which will affect the $a''$ term in the expression of $q''$ in \eqref{eq:q}, we can expect the function $q$ to exhibit a sharp transition from concave to convex in a small interval $(0,\delta)$ to create two solutions of \eqref{eqp33} in this small interval. Explicit calculations using MATLAB show that, taking, e.g., 
	\begin{equation*}
		k = 10^8,
	\end{equation*}
	we have
	\begin{equation*}
		p(0.003) < q(0.003). 
	\end{equation*}
	The conditions in \eqref{eq:at0'} together with $p(0)=q(0)$ imply that $p(v)> q(v)$ for $v\in (0,\epsilon)$ for some sufficiently small $\epsilon>0$. Further, $p(v)\to\infty$ and $q(v)\to 0$ as $v\to\infty$. Thus, equation \eqref{eqp33} has at least one solution in each of the intervals $(0,0.003)$ and $(0.003,\infty)$.
\end{proof}

\begin{rem}
	Repeated modifications of the function $a$ given in \eqref{eq:a} for $t>0$ could lead to examples of $\mathcal{X}\subset K_a$ with $N$ points for larger values of $N$ satisfying the assumption \eqref{eq:r2assumption}.
\end{rem}
\end{appendices}

\bibliographystyle{abbrv}
\bibliography{ref_2D_TN}

\end{document}